\documentclass[reqno,12pt, centertags]{amsart}
\usepackage[left=1in,right=1in,bottom=1in,top=1in]{geometry}
\usepackage{amsmath,amsthm,amscd,amssymb,color,eucal,enumitem,latexsym,lscape,graphicx,multirow,mathrsfs,mathtools}
\mathtoolsset{showonlyrefs}

\usepackage{tikz}
\usepackage{amsmath}
\usepackage{verbatim}
\usetikzlibrary{arrows,shapes}

\newcommand{\bbox}[1]{\vspace{20pt}\fbox{\parbox{450pt}{{\bf #1}}}\vspace{20pt}}
\makeatletter
\def\theequation{\@arabic\c@equation}

\def\namedlabel#1#2{\begingroup
	#2%
	\def\@currentlabel{#2}%
	\phantomsection\label{#1}\endgroup
}

\newcommand{\tA}{\widetilde{A}}

\newcommand{\gaD}{\gamma_{{}_D}}
\newcommand{\gr}{{\text{graph}}}

\newcommand{\gaN}{\gamma_{{}_N}}
\newcommand{\tN}{\tau_{{}_N}}
\newcommand{\Om}{\Omega}
\newcommand{\dOm}{{\partial\Omega}}

\newcommand{\Sp}{\operatorname{Spec}}
\newcommand{\e}{\hbox{\rm e}}
\newcommand{\bd}{\hbox{\rm d}}
\newcommand{\bp}{{\mathbf{p}}}
\newcommand{\bq}{{\mathbf{q}}}

\newcommand{\bbN}{{\mathbb{N}}}
\newcommand{\bbR}{{\mathbb{R}}}

\newcommand{\bbZ}{{\mathbb{Z}}}

\newcommand{\C}{{\mathbb{C}}}

\newcommand{\bbK}{{\mathbb{K}}}

\newcommand{\bbC}{{\mathbb{C}}}

\newcommand{\cA}{{\mathcal A}}
\newcommand{\cB}{{\mathcal B}}

\newcommand{\cD}{{\mathcal D}}
\newcommand{\cE}{{\mathcal E}}
\newcommand{\cg}{{\mathcal G}}
\newcommand{\cG}{{\mathcal F}}
\newcommand{\cH}{{\mathcal H}}
\newcommand{\cI}{{\mathcal I}}
\newcommand{\cJ}{{\mathcal J}}
\newcommand{\cK}{{\mathcal K}}
\newcommand{\cL}{{\mathcal L}}
\newcommand{\cM}{{\mathcal M}}

\newcommand{\cO}{{\mathcal O}}
\newcommand{\cP}{{\mathcal P_{t_0}}}

\newcommand{\cS}{{\mathcal S}}

\newcommand{\cU}{{\mathcal U}}
\newcommand{\cV}{{\mathcal V}}
\newcommand{\cW}{{\mathcal W}}
\newcommand{\cX}{{\mathcal X}}
\newcommand{\cY}{{\mathcal Y}}

\newcommand{\bfi}{{\bf i}}

\newcommand{\no}{\nonumber}
\newcommand{\lb}{\label}

\newcommand{\ol}{\overline}

\newcommand{\wti}{\widetilde  }

\newcommand{\dL}{\prescript{d\!}{}L}

\newcommand{\ran}{\text{\rm{ran}}}

\newcommand{\hatt}{\widehat}

\newcommand{\dist}{\operatorname{dist}}

\newcommand{\rnohs}{\rangle_{H^{-1/2}(\partial\Omega)}}
\newcommand{\lnoh}{_{H^{1/2}(\partial\Omega)}\langle}

\newcommand{\gd}{\widehat{\gamma}_{{}_D}}
\newcommand{\gn}{\widehat{\gamma}_{{}_N}}

\newcommand{\bndr}{H^{{1}/{2}}(\partial \Omega)\times H^{-{1}/{2}}(\partial \Omega)}

\newcommand{\bndra}{\mathfrak{H}\times\mathfrak{H}}

\numberwithin{equation}{section}
\newcommand{\sel}[1]{{\color{blue} Selim:\ \bf{#1}}}
\renewcommand{\div}{\operatorname{div}}
\renewcommand{\det}{\operatorname{det}}

\newcommand{\dom}{\operatorname{dom}}

\newcommand{\tr}{\mathrm{T}}%{\operatorname{\digamma}}%{\operatorname{\mathfrak{t}\mathfrak{r}}}

\newcommand{\spec}{\operatorname{Spec}}

\renewcommand{\Re}{\operatorname{Re }}
\renewcommand\Im{\operatorname{Im}}
\renewcommand{\ker}{\operatorname{ker}}

\newtheorem{theorem}{Theorem}[section]
\newtheorem{hypothesis}[theorem]{Hypothesis}
\newtheorem{lemma}[theorem]{Lemma}
\newtheorem{corollary}[theorem]{Corollary}
\newtheorem{proposition}[theorem]{Proposition}

\theoremstyle{definition}
\newtheorem{definition}[theorem]{Definition}

\newtheorem{example}[theorem]{Example}

\newtheorem{remark}[theorem]{Remark}

\begin{document}

\begin{abstract}In this paper we develop certain aspects of perturbation theory for self-adjoint operators subject to small variations of their domains. We use the abstract theory of boundary triplets to quantify such perturbations and give the second order asymptotic analysis for resolvents, spectral projections, discrete eigenvalues of the corresponding self-adjoint operators. In particular,  we derive explicit formulas for the first variation and the Hessian of the eigenvalue curves bifurcating from a discrete eigenvalue of an unperturbed operator. An application is given to a matrix valued Robin Laplacian and more general Robin-type self-adjoint extensions.    
\end{abstract}

\allowdisplaybreaks

\title[Resolvent expansions]{Resolvent expansions for self-adjoint operators via boundary triplets
}
\author[Y. Latushkin]{Yuri Latushkin}
\address{Department of Mathematics,
The University of Missouri, Columbia, MO 65211, USA}
\email{latushkiny@missouri.edu}
\author[S. Sukhtaiev]{Selim Sukhtaiev	}
\address{Department of Mathematics and Statistics,
Auburn University, Auburn, AL 36849, USA}
\email{szs0266@auburn.edu}
\date{\today}
\thanks{This work naturally emerged from \cite{LSHadamard} and \cite{BCE} as the result numerous stimulating discussions with Prof. G. Berkolaiko. We thank him for generously sharing his ideas as well as his notes on related topics. We wish to thank Prof. P. Kuchment for pointing out many references. Y.L. was supported by the NSF grants  DMS-1710989 and DMS-2106157, and would like to thank the Courant Institute of Mathematical Sciences and especially Prof.\ Lai-Sang Young for the opportunity to visit CIMS}
\maketitle

\section{Introduction}

This paper concerns asymptotic perturbation theory for self-adjoint extensions of symmetric operators. Our main goal is to derive  explicit second order Taylor expansions for resolvents, spectral projections, and eigenvalue curves for one-parameter  families of self-ajdoint extensions of a given symmetric operator. The type of admissible perturbations covers, in particular, variation of boundary conditions for quantum graph operators and, more generally, variation of domain of self-adjointness for extensions of abstract symmetric operators. Such non-additive perturbations arise, for example, in the geometric spectral theory, in which case ad-hoc methods are used to relate, say, the Hessian of parameter-dependent eigenvalues to the reduced Dirichlet-to-Neumann map, see, for example, recent manuscripts \cite{BCCL}, \cite{MR3961373},  \cite{BCE}, \cite{BHMPS}. We provide a complete operator theoretic treatment for the second order expansion problem which, in particular, yields an abstract version of the reduced Dirichlet-to-Neumann map and its relation to the Hessian of eigenvalue curves. The first and second order expansion formulas go back to Hadamard, Rayleigh, and Rellich and have been closely studied subsequently, cf. e.g. \cite{MR3412445, MR2336430, MR2720607, MR2251558, MR2160744}

The first and second order asymptotic expansions discussed in this paper are also central in the Maslov index theory, cf.\  \cite{BBZ18,RoSa95}, for partial differential operators on varying domains or subject to varying boundary conditions, a subject that recently attracted much attention, cf.\ \cite{CJLS,CJM1,CJM2,DJ11,LSHadamard,LS17} and the bibliography therein. For example, Deng--Jones's \cite{DJ11} foundational approach to the spectral count via Maslov index for Robin Laplacian essentially reduces to finding the sign of the derivative of the eigenvalue curves, see recent papers \cite{CJLS, LS17}. Further developments of this approach, see \cite{CJM2}, led to elaborate spectral flow formulas and provided new tools in geometric analysis of Laplace eigenfunctions. 

We stress that the perturbation theory developed here covers not merely ``additive'' perturbations of a given self-adjoint operator but rather a ``structural'' perturbation of the dense subspace on which the perturbed operator is defined. More specifically, we focus on the second order Taylor expansion for resolvents, spectral projections, and eigenvalue curves of $C^2$-regular one parameter families of self-adjoint extensions of a given symmetric operator in abstract Hilbert spaces. We use the abstract theory of boundary triples and Krein--Naimark-type resolvent formulas, a subject having rich and distinguished history \cite{Behrndt_2020, GG, Schm}.  Among applications is an explicit formula for the Hessian and the first derivative of an eigenvalue curve of a one dimensional Schr\"odinger operator with matrix-valued potential subject to varying Robin boundary conditions on a bounded interval.

To give this work yet another perspective, let us mention some closely related papers.   In \cite{CJM1, CJM2}, G. Cox, C. Jones, J. Marzuola utilized explicit formulas for the first derivative of eigenvalue curves of parameter dependent Robin-type elliptic operators to evaluate the spectral shift between the Dirichlet and Neumann realizations in terms of the Morse index of the Dirichlet-to-Neumann map, which yields a plethora of interesting insights in spectral geometry, e.g., a novel formula for the number of nodal domains of Dirichlet eigenfunctions, and a geometric proof Friedlander's inequality, \cite{Fri91}. Two more interesting connections are to a recent paper by G. Berkolaiko, P. Kuchment \cite{BCE}, where the Hessian is computed for simple eigenvalues in the case of ``additive'' perturbations, and to the work of G. Berkolaiko, Y. Canzani, G. Cox, J. L. Marzuola \cite{BCCL}, where the Hessian is computed for dispersion relations (i.e. eigenvalue surfaces parameterized by quasimomenta which arise in the Floquet-Bloch decomposition) of some periodic discrete graphs. It is worth noting that in some of these works, e.g. in \cite{BCCL} the parameter is higher-dimensional while we focus on one-parameter families. That said, we believe that the methods developed in the present paper could be applicable in multi-parameter setting. 

This work should be considered as a sequel to \cite{LSHadamard} where a symplectic form of the Krein resolvent formula was derived in order to establish the first order expansion for resolvents and eigenvalues of self-adjoint extensions. We stress that the abstract setting considered in \cite{LSHadamard} is significantly more general than the one assumed in the current work but the expansion carried out in \cite{LSHadamard} is only to the first order. The major difference is in the type of trace maps used in \cite{LSHadamard} where they are assumed to be merely densely defined and with dense range while in the current work the traces are bounded (in particular, have full domain) and surjective. In short, in \cite{LSHadamard}, $(\mathfrak H, \Gamma_0, \Gamma_1)$ is not required to be a classical boundary triplet, \cite{Behrndt_2020,GG,Schm}, but it is in the current manuscript. 

Let $\cH, \mathfrak{H}$ be complex Hilbert spaces. Let $A$ be a densely defined, closed, symmetric operator acting in $\cH$ and having equal (possibly infinite) deficiency indices. Let $(\Gamma_0, \Gamma_1, \mathfrak H)$ be a boundary triplet associated with $A$, see Hypothesis \ref{hyp3.6}. It is well known that self-adjoint extensions of $A$ are parametrized by Lagrangian planes (also known as self-adjoint linear relations) in $\bndra$, cf., e.g., \cite{Behrndt_2020, GG,  Schm}, by means of the identity $[\Gamma_0, \Gamma_1]^{\top}(\dom(\cA))=\cG$, where $A\subset \cA=\cA^*\subset A^*$ and $\cG\subset \bndra$ is a Lagrangian plane. Firstly, we show that for a $C^2$ path of Lagrangain planes $t\mapsto \cG_t$ and a family of self-adjoint extensions of $A$ with $[\Gamma_0, \Gamma_1]^{\top}(\dom(\cA_t))=\cG_t$,  the path $t\mapsto (\cA_t -\bfi)^{-1}$  is also $C^2$-regular in uniform operator topology. Our argument is based on a Krein-type formula from \cite{LSHadamard}, cf. \eqref{kreinformula}, expressing resolvent difference in terms of the difference of orthogonal projections onto the respective planes.  Secondly, in our main  Theorems \ref{prop1.8new}, \ref{theorem2.2}, \ref{EVExp}, we derive explicit expressions for the coefficients of the second order Taylor expansions for the resolvents of $t-$dependent operators $\cA_t$, the operators per se (more precisely, their conjugated restrictions onto spectral subspaces), and their eigenvalue curves.  Next, in Theorem \ref{prop2.9} we specialized to the case when the Lagrangian planes $\cG_t$ are given as kernels of the $1\times 2$  block operator matrices $Z_t=[X_t, Y_t]$  in $\bndra$, and derive explicit formulas for the first and second derivatives of the eigenvalues of $\cA_t$ in terms of $Z_t$ and its derivatives. In particular,  Corollary \ref{cor211} gives an expansion of the eigenvalue curves for the abstract Robin-type self-adjoint extensions of $A$ up to the second order and demonstrates that the Hessian in this case features a remarkable similarity with that derived by G. Berkolaiko and P. Kuchement in \cite[Eq. 40]{BCE} for additive perturbations of a fixed self-adjoint operator with a simple eigenvalue.   

\section{Resolvent expansions}
Throughout, $A$ denotes a densely defined, symmetric operator acting in a complex Hilbert space $\cH$ and having equal deficiency indices, i.e.,  $\dim\ker(A^*-\bfi)=\dim\ker(A^*+\bfi)\le\infty$.  We equip the vector space $\cH_+=\dom(A^*)$ with the graph scalar product 
\begin{equation}
\langle u, v\rangle_{\cH_+}:=\langle u, v\rangle_{\cH}+\langle A^*u, A^*u\rangle_{\cH}, \ u,v\in\dom(A^*).
\end{equation}
Let $\cH_-:=(\cH_+)^*$ be the dual to $\cH_+$ space consisting of linear bounded functionals. We notice a chain of natural embeddings, 
%\begin{equation}\lb{aub100}
$\cH_+\hookrightarrow\cH\hookrightarrow\cH_-$,
%\end{equation}
where the first embedding is given by $\cH_+\ni u\mapsto u\in\cH$, and the second embedding is given by $\cH\ni v\mapsto \langle\cdot, v\rangle_{\cH}$. Let $\Phi^{-1}:\cH_+\to\cH_-$ be the conjugate linear Riesz isomorphism such that \[{}_{\cH_+}\langle u,\Phi^{-1}w\rangle_{\cH_-}=\langle u,w\rangle_{\cH_+}=
\langle u,w\rangle_{\cH}+\langle A^*u,A^*w\rangle_{\cH}, u,w\in\cH_+.\]

Next, we fix the standard setting of boundary triplets, see \cite{Behrndt_2020, Schm} for more details, the history of the subject and a vast bibliography. Below, we write $[a,b]^{\top}$ for a $2\times 1$ column vector with entries $a,b$.
\begin{hypothesis}\lb{hyp3.6}  Assume that $(\Gamma_0, \Gamma_1, \mathfrak{H})$ is a boundary triplet. That is, we assume that we are given an auxiliary (boundary) Hilbert space $\mathfrak{H}$ and a trace operator
\[\tr:=[\Gamma_0, \Gamma_1]^{\top}\in\cB(\cH_{+}, \bndra),\] which is bounded, surjective and satisfies the following abstract Green identity,
	\begin{equation}\lb{3.61}
	\langle A^*u,v\rangle_{\cH}-\langle u,A^*v\rangle_{\cH}=\langle\Gamma_1u,\Gamma_0v\rangle_{\mathfrak{H}}-\langle\Gamma_0u,  \Gamma_1v\rangle_{\mathfrak{H}} \text{ for all $u,v\in\cH_{+}$}.
	\end{equation}
\end{hypothesis}
Let us introduced a bounded symplectic form $\omega$ defined by the formulas
\begin{align}
\begin{split}\lb{5.3}
&\omega\big((f_1,f_2)^{\top}, (g_1,g_2)^{\top}\big):=\langle f_2,g_1\rangle_{\mathfrak{H}}-\langle f_1, g_2\rangle_{\mathfrak{H}}\\
&\quad=\left\langle
J (f_1, f_2)^{\top},(g_1,g_2)^{\top}
\right\rangle_{\mathfrak{H}\times\mathfrak{H}},\ 
J:=\begin{bmatrix}
0 & I_{\mathfrak{H}} \\
-I_{\mathfrak{H}}& 0
\end{bmatrix},f_k,g_k\in\mathfrak H, k=1,2,
\end{split}
\end{align}
so that the right-hand side of \eqref{3.61} becomes $\omega(\tr u, \tr v)$.
We denote the annihilator of a subspace $\cG\subset\mathfrak{H}\times\mathfrak{H}$ by
\begin{equation}
\cG^{\circ}:=\{ (f_1,f_2)^{\top}\in \mathfrak{H}\times\mathfrak{H}: \omega\big((f_1,f_2)^{\top}, (g_1,g_2)^{\top}\big)=0 \text{\ for all}\ (g_1,g_2)^{\top}\in \cG\},
\end{equation}
and recall that the subspace $\cG$ is called {\it Lagrangian} if $\cG=\cG^{\circ}$. Further, let $\Lambda(\mathfrak{H}\times\mathfrak{H})$ denote the metric space of Lagrangian subspaces of $\mathfrak{H}\times\mathfrak{H}$  equipped with the metric
\begin{equation}\lb{dmet}
d(\cG_1, \cG_2):=\|Q_1-Q_2\|_{\cB(\mathfrak{H}\times\mathfrak{H})},\ \cG_1, \cG_2\in \Lambda(\mathfrak{H}\times\mathfrak{H}), 
\end{equation}
where  $Q_j$ is the orthogonal projection onto $\cG_j$ acting in $\mathfrak{H}\times\mathfrak{H}$, $j=1,2$. 
Let us also recall the classical, and central for this paper, fact that the Lagrangian planes in $\Lambda(\mathfrak{H}\times\mathfrak{H})$ are in one-to-one correspondence with self-adjoint extensions of $A$, see \cite{Schm}. That is, for every self-adjoint extension $\cA$ of $A$ the subspace $\tr(\dom(\cA))\subset \bndra$ is Lagrangian with respect to $\omega$, and, conversely, every Lagrangian plane $\cG\in \Lambda(\bndra)$ gives rise to a self-adjoint extension of $A$ with domain $\tr^{-1}(\cG)\subset \dom(A^*)$, see \cite{LSHadamard} and the references therein for more details. 

In this paper we are interested in explicit second order Taylor expansions for resolvents and eigenvalues of a family of self-adjoint extensions $\cA_t$ of $A$ defined by a given family of Lagrangian planes $t\mapsto \cG_t$ for $t$ near some $t_0$. We note that  the regularity of a path of Lagrangian planes  $\cG_t$ is determined by the regularity of the path $t\mapsto Q_t$ of orthogonal projections $Q_t$ acting in $\bndra$ with $\ran(Q_t)=\cG_t$, in other words, the topology on $\Lambda(\bndra)$ is determined by the metric in \eqref{dmet}.  Throughout, we use ``dot'' to denote $t$-derivative so that, in particular, $\dot{Q}_{t_0}=\frac{\rm d}{{\rm d}t}Q_t\big|_{t=t_0}$. Let us now record our assumptions.

\begin{hypothesis}\label{hyp1.3bis}
Let  $Q: [0,1]\rightarrow \cB(\mathfrak{H}\times\mathfrak{H}), t\mapsto Q_t$ be a one-parameter family of orthogonal projections and assume that  $\ran  (Q_t) \in\Lambda(\mathfrak{H}\times\mathfrak{H})$ is a Lagrangian plane. Fix $t_0\in[0,1]$ and assume that 
\begin{align}
\begin{split}\lb{QVexpansions}
&Q_t\underset{t\rightarrow t_0}{=}Q_{t_0}+\dot Q_{t_0}(t-t_0)+\frac12\ddot Q_{t_0}(t-t_0)^2+o(t-t_0)^2\text{\ in\ }\cB(\bndra).
\end{split}
\end{align}
Let $\cA_t$ be a family of self-adjoint extensions of $A$ satisfying 
	\begin{align}
	{\tr\big( \dom(\cA_{t})\big)}=\ran (Q_t).\lb{3272new-bis}
	\end{align}
Finally, denote $R_t(\zeta):=(\cA_t-\zeta)^{-1}\in\cB(\cH)$ for $\zeta\not\in \Sp(\cA_t)$.
\end{hypothesis}

We recall from \cite[Theorem 3.11]{LSHadamard} the following Krein-type resolvent formula.
\begin{theorem} Assume Hypothesis \ref{hyp1.3bis} and let 
	$t, t_0 \in [0,1],$ $\zeta\not\in\Sp(\cA_{t})\cup \Sp(\cA_{t_0})$. Then	one has
	\begin{align}\lb{kreinformula}
	R_{t}(\zeta)- R_{t_0}(\zeta)=(\tr R_t(\zeta))^*(Q_t-Q_{t_0})JQ_{t_0}\tr R_{t_0}(\zeta). 
	\end{align}
\end{theorem}
\begin{remark}\lb{Remark2.4}
1) Due to Lagrangian properties of $Q_t$ we have $Q_{t_0}JQ_{t_0}=0$; in addition, since $\tr( \ran(R_{t_0}(\zeta)))=\tr (\dom(\cA_{t_0}))=\ran(Q_{t_0})$, we have $Q_{t_0}\tr R_{t_0}(\zeta)=\tr R_{t_0}(\zeta)$ and formula \eqref{kreinformula} can be written as follows,
\begin{align}\lb{krein1.0}
R_{t}(\zeta)- R_{t_0}(\zeta)=(\tr R_t(\zeta))^*Q_tJ\tr R_{t_0}(\zeta). 
\end{align}

2) The adjoint operator $(\tr R_t(\zeta))^*$ in the right-hand side of \eqref{kreinformula} is a rather subtle object. It should be viewed as an operator in $\cB(\bndra, \cH)$. Indeed, one has $\tr\in\cB(\cH_+, \bndra)$ and  $\ran(R_t(\zeta))\subset\cH_+$, so that the product $\tr (R_t(\zeta))$ is well defined as a mapping from $\cH$ to $\bndra$. Moreover, we claim that $R_t(\zeta)\in\cB(\cH, \cH_+)$, hence, in particular $\tr (R_t(\zeta))\in\cB(\cH, \bndra)$ so its adjoint is indeed in $\cB(\bndra, \cH)$ as claimed above. To check the inclusion $R_t(\zeta)\in\cB(\cH, \cH_+)$ we notice that  for some $c>0$ and every $u\in\cH$ one has
\begin{align}
\|R_{t_0}(\zeta)u\|^2_{\cH_+}&=\|R_{t_0}(\zeta)u\|^2_{\cH}+\|A^*R_{t_0}(\zeta)u\|^2_{\cH}\\
&\leq \|R_{t_0}(\zeta)u\|^2_{\cH}+2\|(\cA_{t_0}-\zeta)R_{t_0}(\zeta)u\|^2_{\cH}+2\|\zeta R_{t_0}(\zeta)u\|^2_{\cH}\\
&\leq \|R_{t_0}(\zeta)u\|^2_{\cH}+2\|u\|^2_{\cH}+2\|\zeta R_{t_0}(\zeta)u\|^2_{\cH}\leq c\|u\|^2_{\cH}. 
\end{align}

3) It is possible to represent $(\tr R_t(\zeta))^*$ as the product of two adjoints, $(R_t(\zeta))^*$ and $\tr^*$. In fact, the adjoint of the resolvent $(R_t({\zeta}))^*\in\cB(\cH_-,\cH)$ is a natural extension of $R_t( \overline{\zeta})$, originally defied as a mapping from $\cH$ to itself, to a mapping from $\cH_-$ to $\cH$, see \cite[Proposition 2.4]{LSHadamard}. In this work, however, we chose not to unfold the product $(\tr R_t(\zeta))^*$ and shall view $R_t(\zeta)$ as an element of $\cB(\cH, \cH_+)$ whenever it precedes $\tr$. 

4) From the outset we only have $\ran((\tr R_t(\zeta))^*)\subset \cH$. However, \eqref{krein1.0} suggests that $\ran((\tr R_t(\zeta))^*)\subset \cH_+$  as both resolvents in the left-hand side of \eqref{krein1.0} have ranges in $\cH_+$. It turns out, cf. Theorem \ref{prop1.8new} (3), that Krein's-type formula \eqref{krein1.0} produces a similar effect for the derivative  $\dot R_{t_0}(\zeta)$ that satisfies $\ran(\dot R_{t_0}(\zeta))\subset \cH_+$  although {\it a priori} one only has $\ran(\dot R_{t_0}(\zeta))\subset \cH$.
\end{remark}

Our first principal result gives Taylor expansions for the resolvent of $R_t(\zeta)$
\begin{theorem}\lb{prop1.8new} Assume Hypothesis \ref{hyp1.3bis}.
Suppose that $\zeta_0\not \in \spec(\cA_{t_0})$ for some $t_0\in[0,1]$ and $\zeta_0\in\bbC$, and define 
 	\begin{align}
 &\cU_\varepsilon=\{(t,\zeta)\in[0,1]\times\bbC: |t-t_0|\le\epsilon, |\zeta-\zeta_0|\le\epsilon\} \text{ for $\epsilon>0$}.
 \end{align}  
	\begin{enumerate}
		\item\lb{18i1}There exists an $\varepsilon>0$ such that if $(t,\zeta)\in\cU_\varepsilon$ then $\zeta\not\in\Sp(\cA_t)$.
		\item For $\varepsilon>0$, $(t, \zeta)\in\cU_{\varepsilon}$ as in part \eqref{18i1} above one has
		\begin{equation}\lb{bigop}
		(\tr R_t(\zeta))^*(Q_t-Q_{t_0})JQ_{t_0}\tr R_{t_0}(\zeta)\in\cB(\cH, \cH_+).
		\end{equation}
		\item For $\varepsilon>0$, $\zeta\in\bbC$ as in  part \eqref{18i1} above  we define
		\begin{align}
		&\dot R_{t_0}(\zeta):=  (\tr R_{t_0}(\zeta))^*\dot Q_{t_0}J\tr R_{t_0}(\zeta).\lb{derR} 
		\end{align}
		Then one has 
		\begin{align}
		&\dot R_{t_0}(\zeta)\in \cB(\cH, \cH_+).\lb{dbound}
		\end{align}
		Moreover, the following asymptotic expansion holds uniformly for $|\zeta-\zeta_0|<\varepsilon$,
		\begin{align}\lb{vtor}
		&R_t(\zeta)=R_{t_0}(\zeta)+\dot R_{t_0}(\zeta)(t-t_0)+o(t-t_0) \text{\ in\ } \cB(\cH, \cH_+).
		\end{align}
		\item\lb{18i3}
		For $\varepsilon>0$ as in part \eqref{18i1}  the following asymptotic expansion holds uniformly for $|\zeta-\zeta_0|<\varepsilon$,
		\begin{align}
		&R_t(\zeta)\underset{t\rightarrow t_0}{=}R_{t_0}(\zeta)+\dot R_{t_0}(\zeta)(t-t_0)+\frac{1}{2} \ddot R_{t_0}(\zeta)(t-t_0)^2+o(t-t_0)^2\ \text{\ in\ }\cB(\cH),\lb{new1.43nn}
		\end{align}
		where $\dot R_{t_0}(\zeta)$ is given by \eqref{derR} and
		\begin{align}
		\ddot R_{t_0}(\zeta):=& (\tr R_{t_0})^*\ddot Q_{t_0}JTR_{t_0}+2(\tr\dot R_{t_0}(\zeta))^*\dot Q_{t_0}J\tr R_{t_0}(\zeta).\lb{dderR}
		\end{align}
	\end{enumerate}
\end{theorem}

\begin{proof}% Our strategy is as follows.
%In order to obtain the second order Taylor expansion for   $t\mapsto R_t(\zeta)\in\cB(\cH)$ we derive below an expansions for $t\mapsto R_t(\zeta)(V_{t_0}-V_t)\in\cB(\cH)$, $t\mapsto R_t(\zeta)(K_t-K_{t_0})\in\cB(\cH)$ and substitute them into Krein's-type formula \eqref{kreinformula}. Since $K_t-K_{t_0}$ is of order $t-t_0$, it is enough to expand $t\mapsto R_t(\zeta)$ in $t\mapsto R_t(\zeta)(K_t-K_{t_0})$ up to the first order. The expansion for $t\mapsto R_t(\zeta)$, however, must be carried out in $\cB(\cH_-,\cH)$ or, equivalently, in the dual space $\cB(\cH, \cH_+)$,  since $(K_t-K_{t_0})\in\cB(\cH_+,\cH_-)$. This subtlety is addressed by bootstrapping resolvent norm estimates initially obtained in $\cB(\cH)$ to those in $\cB(\cH, \cH_+)$. 

{\it Proof of} \eqref{18i1}. First, we show that for $\zeta=\bfi$ we have
\begin{equation}\lb{bddRes}
\|R_t(\bfi)\|_{\cB(\cH,\cH_+)}\underset{t\rightarrow t_0}{=}\cO(1).
\end{equation}
Of course, such asymptotic formula holds in $\cB(\cH)$ by self-adjointness of $\cA_t$. A stronger version which we are after holds due to the fact that $R_t(\bfi)$ is the resolvent of a self-adjoint restriction of $A^*$; this observation is crucial throughout the proof. 

Since  $\|R_t(\bfi)\|_{\cB(\cH)}\leq 1$ there exists $c>0$ such that for all $u\in\cH_+$ and $t$ near $t_0$ one has
\begin{align}
\begin{split}\lb{arg1}
\|R_t(\bfi)u\|_{\cH_+}^2&= \|R_t(\bfi)u\|_{\cH}^2+\|A^*R_t(\bfi)u\|_{\cH}^2\\
&= \|R_t(\bfi)u\|_{\cH}^2+2\|(A_t-\bfi) R_t(\bfi)u\|_{\cH}^2+2\|\bfi R_t(\bfi)u\|_{\cH}^2 \leq c\|u\|_{\cH}^2
\end{split}
\end{align}
which confirms \eqref{bddRes} (we used $(\cA_t-\bfi) R_t(\bfi)=I$). Next, by Krein's formula \eqref{krein1.0} with $\zeta=\bfi$, we have
\[
R_t(\bfi)=R_{t_0}(\bfi)+ (\tr R_t(\bfi))^*(Q_t-Q_{t_0})J\tr R_{t_0}(\bfi).
\]
Combining this with $ \|\tr R_t(\bfi)\|_{\cB(\cH,\bndra)}\underset{t\rightarrow t_0}{=}\cO(1)$ we get 
\begin{equation}\lb{resnepr}
\|R_t(\bfi)-R_{t_0}(\bfi)\|_{\cB(\cH)}\underset{t\rightarrow t_0}{=}\cO(t-t_0),
\end{equation}
so that $\cA_t\rightarrow \cA_{t_0}$, $t\rightarrow t_0$ in the norm resolvent sense.  Then \cite[Theorem VIII.23(a)]{RS1} yields the assertion.

{\it Proof of (2).} By Krein's formula the operator in \eqref{bigop} is equal to the difference of the resolvents of two self-adjoint restrictions of $A^*$, each of which is an element of $\cB(\cH, \cH_+)$ as discussed in Remark \ref{Remark2.4} part 2).

{\it Proof of (3).} We begin by obtaining an improved version of  \eqref{resnepr} given by
\begin{equation}\lb{1resnepr}
\|R_t(\bfi)-R_{t_0}(\bfi)\|_{\cB(\cH, \cH_+)}\underset{t\rightarrow t_0}{=}\cO(t-t_0).
\end{equation}
To that end we write
\begin{align}\lb{3291ineq}
&\|R_t(\bfi)u-R_{t_0}(\bfi)u\|_{\cH_+}^2 = \|R_t(\bfi)u-R_{t_0}(\bfi)u\|_{\cH}^2+\|A^*R_t(\bfi)u-A^*R_{t_0}(\bfi)u\|^2_{\cH}\\
&\quad\leq  \|R_t(\bfi)u-R_{t_0}(\bfi)u\|_{\cH}^2+2\|(\cA_t-\bfi )R_t(\bfi)u-(\cA_{t_0}-\bfi )R_{t_0}(\bfi)u\|^2_{\cH}\\
&\hspace{5.1cm}+2\|-\bfi R_t(\bfi)u+\bfi R_{t_0}(\bfi)u\|^2_{\cH}\\
&\quad\leq  3\|R_t(\bfi)u-R_{t_0}(\bfi)u\|_{\cH}^2\underset{t\rightarrow t_0}{=}\cO(t-t_0), \text{\ uniformly for\ } \|u\|_{\cH}\leq 1.
\end{align}
We will need a slightly more general version of \eqref{1resnepr} in which $\bfi$ is replaced by $\zeta\in\bbC$ near $\zeta_0$.  Let us fix $\varepsilon_0>0$ such that $\mathbb{B}_{\varepsilon_0}(\zeta_0)\subset \C\setminus \Sp(\cA_{t_0})$ for the $\varepsilon_0$-disc centered at $\zeta_0$. Then by \eqref{resnepr} and \cite[Theorem VIII.23]{RS1} we have
$\mathbb{B}_{\varepsilon_0}(\zeta_0)\cap\Sp(\cA_{t})=\emptyset$
for $t$ sufficiently close to $t_0$. Hence, 
\begin{equation}\lb{new1.17}
\sup\{\| R_t(\zeta)\|_{\cB(\cH)}: (t,\zeta)\in \cU_{\varepsilon}\}<\infty
\end{equation}
for a sufficiently small $\varepsilon\in(0,\varepsilon_0)$. Then combining the resolvent identity
\begin{equation}
R_t(\zeta)=	  R_t(\bfi)-(\bfi-\zeta)  R_t(\bfi)R_t(\zeta),
\end{equation}
and \eqref{new1.17} we infer that the  asymptotic formula
\begin{align}
&\|R_t(\zeta)-R_{t_0}(\zeta)\|_{\cB(\cH, \cH_+)}\underset{t\rightarrow t_0}{=}\cO(t-t_0),
\end{align}
holds uniformly for $|\zeta-\zeta_0|<\varepsilon$.

We are now ready to turn to the proof of \eqref{dbound}, \eqref{vtor}. To that end we will first establish the same expansion but in the space $\cB(\cH)$ possessing weaker topology. Employing Krein's formula \eqref{Remark2.4} ones again we obtain
\begin{align}
\begin{split}\lb{frstord}
&R_{t}(\zeta)\underset{t\rightarrow t_0}{=}
R_{t_0}(\zeta)+\big(TR_{t_0}(\zeta)
+\cO_{\cB(\cH, \bndra)}(t-t_0)\big)^*\big(\dot Q_{t_0}(t-t_0)+o(t-t_0)\big)J\tr R_{t_0}(\zeta)\\
& =R_{t_0}(\zeta)+\dot R_{t_0}(\zeta)(t-t_0)+o(t-t_0) \text{\ in\ } \cB(\cH).
\end{split}
\end{align}
Our next objective is to derive this asymptotic expansion in $\cB(\cH, \cH_+)$ and to prove \eqref{dbound}. To that end we pick any sequence $t_n\rightarrow t_0$, $n\rightarrow\infty$ and introduce a sequence of operators
\begin{align}
\cD_{n}(\zeta):=\frac{R_{t_n}(\zeta)-R_{t_0}(\zeta)}{{t_n}-t_0}, n\in\bbN.
\end{align}
First we notice that $\ran(\cD_{n})\subset\dom(A^*)=\cH_+$ and, in fact, by Remark \ref{Remark2.4} (2) one has $\cD_{n}\in\cB(\cH, \cH_+)$. We claim that the sequence $\{\cD_n(\zeta)\}$ is Cauchy in $\cB(\cH, \cH_+)$. Let us observe that
\begin{align}
A^*(R_{t_n}(\zeta)-R_{t_m}(\zeta))=\zeta(R_{t_m}(\zeta)-R_{t_n}(\zeta)),
\end{align}
that is $A^*(\cD_n(\zeta)-\cD_m(\zeta))=\zeta(\cD_m(\zeta)-\cD_n(\zeta))$, so for arbitrary $u\in\cH$ we have
\begin{align}
&\|(\cD_{n}(\zeta)-\cD_{m}(\zeta))u\|_{\cH_+}^2=\left\|(\cD_{n}(\zeta)-\cD_{m}(\zeta))u \right\|^2_{\cH}+\left\|A^*(\cD_{n}(\zeta)-\cD_{m}(\zeta))u \right\|^2_{\cH}\\
&=(1+|\zeta|^2)\|(\cD_n(\zeta)-\cD_m(\zeta))u \|^2_{\cH},
\end{align}
thus, one has
\begin{align}
&\|\cD_{n}(\zeta)-\cD_{m}(\zeta)\|_{\cB(\cH, \cH_+)}^2\leq (1+|\zeta|^2)\|\cD_n(\zeta)-\cD_m(\zeta) \|^2_{\cB(\cH)}\underset{m,n\rightarrow \infty}{\rightarrow} 0,
\end{align}
where we used that $\{\cD_{n}\}_{n\in\bbN}$ is Cauchy in $\cB(\cH)$ by \eqref{frstord}, \eqref{resnepr}.  Thus,
the sequence $\cD_{n}$ converges in $\cB(\cH, \cH_+)$ to an operator  $D_{\infty}\in\cB(\cH, \cH_+)$. By definition of $\|\cdot\|_{\cH_+}$, we automatically also infer 
\begin{equation}
\|\cD_{n}(\zeta)-\cD_{\infty}(\zeta)\|_{\cB(\cH)}\rightarrow 0, n\rightarrow\infty. 
\end{equation}
At the same time,  \eqref{frstord} show that
\begin{equation}
\|\cD_{n}(\zeta)-\dot R_{t_0}(\zeta)\|_{\cB(\cH)}\rightarrow 0, n\rightarrow\infty. 
\end{equation}
Hence $\dot R_{t_0}(\zeta)=\cD_{\infty}(\zeta)\in\cB(\cH, \cH_+)$ as asserted in\eqref{dbound}. Also, $\cD_{\infty}(\zeta)$ is independent of the choice of $t_n\rightarrow t_0$. Thus
\begin{equation}
R_t(\zeta)-R_{t_0}(\zeta)= \cD_{\infty}(\zeta)(t-t_0)+o(t-t_0),\ t\rightarrow t_0,
\end{equation} 
and \eqref{vtor} holds. 

{\it Proof of \eqref{18i3}.} The asymptotic formula \eqref{new1.43nn} is obtained by substituting resolvent expansion \eqref{vtor} into Krein's formula \eqref{kreinformula} and grouping powers of $(t-t_0)$. This formal expansion is indeed justified since the operator $\tr \dot R_{t_0}(\zeta)$  is well defined and, by \eqref{vtor}, we have
\begin{align}
&\tr R_t(\zeta)=\tr R_{t_0}(\zeta)+\tr \dot R_{t_0}(\zeta)(t-t_0)+o(t-t_0) \text{\ in\ } \cB(\cH, \bndra).
\end{align}
\end{proof}

In what follows we will describe how a multiple isolated eigenvalue $\lambda$ of $\cA_{t_0}$ bifurcates for $t$ near $t_0$.

\begin{hypothesis}\lb{hyp3.1o}
	Fix $t_0\in[0,1]$, suppose that $\lambda=\lambda(t_0)$ is an isolated eigenvalue of $\cA_{t_0}$ with finite multiplicity $m\in\bbN$. Let 
	\begin{equation}\no
	\gamma:=\big\{z\in\bbC: 2|z-\lambda|=\dist\big(\lambda, \Sp(\cA_{t_0})\setminus\{\lambda\}\big)\big\}.
	\end{equation}
	%and let $B\subset\bbC$ denote the disc enclosed by $\gamma$. 
\end{hypothesis}

We continue to assume Hypothesis \ref{hyp1.3bis}. Then, by Theorem \ref{prop1.8new}, there exists $\varepsilon>0$ such that $\gamma$ encloses the $\lambda$-group of $m$ eigenvalues (not necessarily distinct) of the operator $\cA_t$ whenever $|t-t_0|<\varepsilon$ and $\varepsilon>0$ is sufficiently small.  For such $t$ we let $P_t$ denote the Riesz projection for $\cA_t$ corresponding to the $\lambda$-group,
\begin{equation}\lb{3.1}
P_t:=\frac{-1}{2\pi \bfi}\int_{\gamma} R_t(\zeta)d\zeta, \, \text{ where $R_t(\zeta)=(\cA_t-\zeta)^{-1}$}.
\end{equation}  
Let us recall from \cite[Section III.6.5]{K80} the following Laurent expansion for the resolvent,
\begin{equation}\lb{3.1ab}
R_{t_0}(\zeta)=(\lambda-\zeta)^{-1}P_{t_0}+\sum_{n=0}^{\infty}(\zeta-\lambda)^n S^{n+1}_{t_0}, 
\end{equation}
where $S_{t_0}=\lim_{\zeta\to\lambda}R_{t_0}(\zeta)(I_\cH-P_{t_0})$ is the reduced resolvent of the operator $\cA_{t_0}$ at the eigenvalue $\lambda$ given by the formula
\begin{equation}\lb{3.1a}
S_{t_0}:=\frac{-1}{2\pi \bfi}\int_{\gamma}(\lambda-\zeta)^{-1} R_{t_0}(\zeta)d\zeta.						
\end{equation}	 

\begin{remark} We note that by Theorem \ref{prop1.8new}, part (3) one has 
\begin{equation}(\tr R_{t_0}(\zeta))^*\dot Q_{t_0}J\tr R_{t_0}(\zeta)\in\cB(\cH, \cH_+).
\end{equation}
Therefore, in particular
\begin{align}
\tr (\tr R_{t_0}(\zeta))^*\dot Q_{t_0}J\tr R_{t_0}(\zeta)\in \cB(\cH,\bndra).
\end{align}
Furthermore, since the integrals in \eqref{3.1} and \eqref{3.1a} converge with respect to the norm of the space $\cB(\cH, \cH_+)$, the  two inclusions above also hold with the two copies of $R_{t_0}(\zeta)$ replaced by $S_{t_0}^2$, $S_{t_0}$, $P_{t_0}$ in any combination. In particular, the operators in \eqref{pone}, \eqref{ptwo}, \eqref{T1}, \eqref{T2}, \eqref{XYsecodr}, and in Remark \ref{rem2.14} used below are all well defined and bounded in the respective spaces. 
\end{remark}

The following result gives the second order Taylor expansion for the Riesz projection just introduced.
\begin{proposition}\lb{Proposition1.5} Under Hypotheses \ref{hyp1.3bis} and \ref{hyp3.1o}, for the Riesz projection introduced in \eqref{3.1} one has
\begin{align}
	&P_t\underset{t\rightarrow t_0}{=}P_{t_0}+\dot P_{t_0}(t-t_0)+\frac{1}{2} \ddot P_{t_0}(t-t_0)^2+o(t-t_0)^2,\ \text{\ in\ }\cB(\cH),\lb{new1.43nn2}
\end{align}
with 
\begin{align}
\dot P_{t_0}&:=  (\tr S_{t_0})^*\dot Q_{t_0}J\tr P_{t_0}+(\tr P_{t_0})^*\dot Q_{t_0}J\tr S_{t_0},\lb{pone}
\end{align}
and
\begin{align}
\begin{split}\lb{ptwo}
\ddot P_{t_0}:=&(\tr S_{t_0})^*\ddot Q_{t_0}J\tr P_{t_0}+(\tr P_{t_0})^*\ddot Q_{t_0}J\tr S_{t_0}\\
&+2\big(\tr (\tr S_{t_0})^*\dot Q_{t_0}J\tr S_{t_0}\big)^*\dot Q_{t_0}J\tr P_{t_0}+2\big(\tr (\tr S_{t_0})^*\dot Q_{t_0}J\tr P_{t_0}\big)^*\dot Q_{t_0}J\tr S_{t_0}\\
&+2\big(\tr (\tr P_{t_0})^*\dot Q_{t_0}J\tr S_{t_0}\big)^*\dot Q_{t_0}J\tr S_{t_0}-2\big(\tr (\tr S_{t_0}^2)^*\dot Q_{t_0}J\tr P_{t_0}\big)^*\dot Q_{t_0}J\tr P_{t_0}\\
&-2\big(\tr (\tr P_{t_0})^*\dot Q_{t_0}J\tr S_{t_0}^2\big)^*\dot Q_{t_0}J\tr P_{t_0}-2\big(\tr (\tr P_{t_0})^*\dot Q_{t_0}J\tr P_{t_0}\big)^*\dot Q_{t_0}J\tr S_{t_0}^2.
\end{split}
\end{align}
\end{proposition}
\begin{proof} 
	We plug \eqref{3.1ab} in \eqref{derR}, \eqref{dderR} and observe that the coefficients of the terms containing  $(\lambda-\zeta)^{-1}$  in the Laurent expansions of $\dot{R}_{t_0}(\zeta)$, $\ddot{R}_{t_0}(\zeta)$ are given by the right-hand sides of \eqref{pone}, \eqref{ptwo} respectively. Upon using the right hand sides of  \eqref{pone}, \eqref{ptwo}, in the Laurent  expansions of $\dot R_{t_0}(\zeta)$, $\ddot R_{t_0}(\zeta)$, we multiply both sides of \eqref{new1.43nn} by $\frac{-1}{2\pi \bfi}$, integrate over $\gamma$, and employ \eqref{3.1} together with Cauchy's integral formula to arrive at \eqref{new1.43nn2}.  In other words, $\dot{P}_{t_0}$, $\ddot{P}_{t_0}$ are residues of the operator-valued functions $\zeta\mapsto -\dot{R}_{t_0}(\zeta)$, $\zeta\mapsto-\ddot{R}_{t_0}(\zeta)$; the residues are given by \eqref{pone}, \eqref{ptwo}.
\end{proof}
Our next objective is to derive an asymptotic expansion for the operator that is similar to $P_t\cA_tP_t$ for $t$ near $t_0$ but acts on a $t$-independent  $m-$dimensional space $\ran(P_{t_0})$. To that end, we introduce the operator $U_t$  mapping $\ran(P_{t_0})$ onto $\ran(P_t)$ unitarily, for $t$ near $t_0$, and defined by the formulas  
\begin{align}
\begin{split}\lb{int27}
	&U_t:=(I-D^2_t)^{-1/2}\big((I-P_t)(I-P_{t_0})+P_tP_{t_0}\big),\\
	&U_t^{-1}=\big((I-P_{t_0})(I-P_t)+P_{t_0}P_t\big)(I-D^2_t)^{-1/2},
\end{split}
\end{align}
where $D_t:=P_t-P_{t_0}$ is invertible (for $t$ near $t_0$) since $\|D_t \|_{\cB(\cH)}\underset{t\rightarrow t_0}{=}o(1)$ by \eqref{new1.43nn2}. In particular, one has $U_tP_{t_0}=P_tU_t$. Given the auxiliary operators $U_t$ we consider an $m-$dimensional operator $W_t:=P_{t_0}U_t^{-1}\cA_tU_tP_{t_0}$ and obtain its second order Taylor expansion next. Here and in what follows we identify $W_t$ and its restriction to $\ran(P_{t_0})$. 
\begin{theorem}\lb{theorem2.2} Assume Hypotheses \ref{hyp1.3bis} and \ref{hyp3.1o}  and let $W_t:=P_{t_0}U_t^{-1}\cA_tU_tP_{t_0}$ for $t$ near $t_0$. Then the following second order Taylor expansion holds, 
\begin{align}
\begin{split}
\lb{1.42n}
W_t\underset{t\rightarrow t_0}{=}W_{t_0}+\dot W_{t_0}(t-t_0)+\frac12\ddot W_{t_0}(t-t_0)^2+o(t-t_0)^2,
\end{split}
\end{align}
where $W_{t_0}:=\lambda P_{t_0}$ and 
\begin{align}
&\dot W_{t_0}:=(TP_{t_0})^*J\dot Q_{t_0} TP_{t_0},\lb{T1}\\
&\ddot W_{t_0}:=(TP_{t_0})^*J\ddot Q_{t_0} TP_{t_0}-2\big(\tr (\tr S_{t_0})^*\dot Q_{t_0}J\tr P_{t_0}\big)^*\dot Q_{t_0}J\tr P_{t_0}.\lb{T2}
\end{align}
\end{theorem}

\begin{proof} 
	Our strategy is to expand the function
	\begin{equation}\lb{sim}
	t\mapsto P_{t_0}U_t^{-1}R_t(\zeta)U_tP_{t_0},
	\end{equation} 
	uniformly for $\zeta$ near $\lambda$, to the second order in $(t-t_0)$, then multiply both sides of the obtained expansion by $\frac{-1}{2\pi\bfi}\zeta$, integrate over $\gamma$, and use the formula
\begin{equation}\lb{hpt}
\cA_tP_t=\frac{-1}{2\pi \bfi}\int_{\gamma} \zeta R_t(\zeta)d\zeta.
\end{equation}
In order to obtain the expansion for the operator in \eqref{sim}, we split it into four terms,   
\begin{align}\begin{split}
&P_{t_0}U_t^{-1}R_t(\zeta)U_tP_{t_0}=\\
\begin{split}
&=P_{t_0}U_t^{-1}P_{t_0}R_t(\zeta)P_{t_0}U_tP_{t_0}+P_{t_0}U_t^{-1}(I-P_{t_0})R_t(\zeta)P_{t_0}U_tP_{t_0}\\
&\quad+P_{t_0}U_t^{-1}P_{t_0}R_t(\zeta)(I-P_{t_0})U_tP_{t_0}+P_{t_0}U_t^{-1}(I-P_{t_0})R_t(\zeta)(I-P_{t_0})U_tP_{t_0}\lb{fourterms}
\end{split}\\
&=\tau_1(\zeta)+\tau_2(\zeta)+\tau_3(\zeta)+\tau_4(\zeta),
\end{split}
\end{align}
where $\tau_k(\zeta)$ denotes the respective term in \eqref{fourterms}.
The remaining part of the argument is divided into three steps. In Steps 1 and 2, we derive auxiliary expansions in $t-t_0$. Step 3  gives asymptotic expansion for integrated versions of $\tau_k(\zeta)$.

{\bf Step 1.} Let us denote
\begin{align}\lb{cs}
\begin{split}
&\mathfrak S_0:=(\tr P_{t_0})^*\dot Q_{t_0}J\tr S_{t_0},\\
 \mathfrak S_1:=\big(\tr (\tr S_{t_0})^*\dot Q_{t_0}J\tr P_{t_0}\big)^*&\dot Q_{t_0}J\tr P_{t_0}, \mathfrak S_2: =\big(\tr (\tr S_{t_0}^2)^*\dot Q_{t_0}J\tr P_{t_0}\big)^*\dot Q_{t_0}J\tr P_{t_0}. 
\end{split}
\end{align}
Then one has
	\begin{align}
	P_{t_0}P_tP_{t_0}&=P_{t_0}-\mathfrak S_2(t-t_0)^2+o(t-t_0)^2,\lb{ptp}\\
	P_{t_0}P_t&=P_{t_0}+\mathfrak S_0(t-t_0)+o(t-t_0).\lb{pt}
	\end{align}
	\begin{proof}[Proof of \eqref{ptp}, \eqref{pt}.]
	To prove \eqref{ptp}, we multiply both sides of \eqref{new1.43nn2} by $P_{t_0}$ from the left and right sides and use  \eqref{pone}, \eqref{ptwo}, and the identities $S_{t_0}P_{t_0}=0$ and $P_{t_0}(\tr S_{t_0})^*=(\tr S_{t_0}P_{t_0})^*$. Similarly, one obtains \eqref{pt}. 
	\end{proof}
	
	{\bf Step 2.} One has 
	\begin{align}
	& P_{t_0}U_tP_{t_0}=P_{t_0}-\frac{1}{2}\mathfrak S_2(t-t_0)^2+o(t-t_0)^2,\lb{pup}\\
	&P_{t_0}U^{-1}_t(I-P_{t_0})=\mathfrak S_0(t-t_0)+o(t-t_0),\lb{pu1p}
	\end{align}
	where we used notation \eqref{cs}. 
	\begin{proof}[Proof of \eqref{pup}, \eqref{pu1p}]
	Let us derive \eqref{pup}. First, recalling $U_t$ from \eqref{int27} we get
	\begin{align}
	P_{t_0}U_tP_{t_0}&=P_{t_0}(I-D^2_t)^{-1/2}P_tP_{t_0}=P_{t_0}(I+\frac{1}{2}D^2_t+o(t-t_0)^2)P_tP_{t_0}\\
	&=P_{t_0}P_tP_{t_0}+\frac{1}{2}P_{t_0}D^2_tP_tP_{t_0}+o(t-t_0)^2.\lb{cuteq}
	\end{align}
	Next, since $D_t=P_t-P_{t_0}=\dot P_{t_0}(t-t_0)+o(t-t_0)$, we have
	\begin{align}
	\begin{split}\lb{pdsp}
	&P_{t_0}D^2_tP_tP_{t_0}=P_{t_0}(\dot P_{t_0}^2(t-t_0)^2+o(t-t_0)^2)(P_{t_0}+\dot P_{t_0}(t-t_0)+o(t-t_0))P_{t_0}\\
	&=P_{t_0}\dot P_{t_0}^2P_{t_0}(t-t_0)^2+o(t-t_0)^2=\mathfrak S_2(t-t_0)^2+o(t-t_0)^2,
	\end{split}
	\end{align}
	where in the last equality we utilized \eqref{pone} and $S_{t_0}P_{t_0}=0$.
	 Then \eqref{pup} can be inferred from \eqref{ptp}, \eqref{cuteq}, \eqref{pdsp}.

To prove \eqref{pu1p}, we use  \eqref{int27}, \eqref{pt} and get
	\begin{align}
	P_{t_0}U^{-1}_t(I-P_{t_0})&=P_{t_0}(I-P_{t_0}-P_t+2P_{t_0}P_t+\cO(t-t_0)^2)(I-P_{t_0})\\
	&=P_{t_0}P_t (I-P_{t_0})+o(t-t_0)=\mathfrak S_0(t-t_0)+o(t-t_0),
	\end{align}
	where we used $\mathfrak S_0(1-P_{t_0})=\mathfrak S_0$. 
	\end{proof}
	{\bf Step 3.} 
Let us denote \begin{align}
&\mathfrak Q_1: =(\tr P_{t_0})^*\dot Q_{t_0}J\tr P_{t_0},\  \mathfrak Q_2: =(\tr P_{t_0})^*\ddot Q_{t_0}J\tr P_{t_0},
\end{align}
recall $\tau_k(\zeta)$ introduced in \eqref{fourterms}, and $\mathfrak S_0$, $\mathfrak S_1$, $\mathfrak S_2$ from \eqref{cs}. Then we have
\begin{align}\lb{A1}
&\frac{-1}{2\pi \bfi}\int_{\gamma} \zeta \tau_1(\zeta)d\zeta=\lambda P_{t_0}-\mathfrak Q_1(t-t_0)-\Big(\frac{1}{2} \mathfrak Q_2+2\lambda\mathfrak S_2+\mathfrak S_1\Big)(t-t_0)^2+o(t-t_0)^2,\\\lb{A2}
&\frac{-1}{2\pi \bfi}\int_{\gamma} \zeta \tau_k(\zeta)d\zeta=\lambda\mathfrak S_2(t-t_0)^2+o(t-t_0)^2,\ k=2,3,\\
& \frac{-1}{2\pi \bfi}\int_{\gamma} \zeta \tau_4(\zeta)d\zeta=o(t-t_0)^2.\lb{A4}
\end{align}
\begin{proof}[Proof of \eqref{A1}, \eqref{A2}]
To prove \eqref{A1} we will combine the second order Taylor expansions \eqref{pup} and \eqref{new1.43nn} with respect to parameter $t$, collect powers of $(\lambda-\zeta)$, and contour integrate with respect to $\zeta$ as follows. 

Introducing $\tilde{\mathfrak R}(\zeta):=(\tr (\tr R_{t_0}(\zeta))^*\dot Q_{t_0}J\tr P_{t_0})^*\dot Q_{t_0}J\tr P_{t_0}$ and noting  that by conjugation $P_{t_0}U_t^{-1}P_{t_0}$ admits the same asymptotic as $P_{t_0}U_tP_{t_0}$), we infer
\begin{align}
&\tau_1(\zeta)=P_{t_0}\Big(P_{t_0}-\frac{1}{2}\mathfrak S_2(t-t_0)^2+o(t-t_0)^2\Big)P_{t_0}\times\\
&\times P_{t_0}\Big((\lambda-\zeta)^{-1}P_{t_0}+(\lambda-\zeta)^{-2}\mathfrak Q_1(t-t_0) +\big(\frac{1}{2}(\lambda-\zeta)^{-2}\mathfrak Q_2+(\lambda-\zeta)^{-2}\tilde {\mathfrak R}(\zeta)\big)(t-t_0)^2\Big)P_{t_0}\times\\
&\times\Big(P_{t_0}-\frac{1}{2}\mathfrak S_2(t-t_0)^2+o(t-t_0)^2\Big)P_{t_0},
\end{align}
hence,
\begin{align}
\begin{split}\lb{tau1}
&\tau_1(\zeta)=(\lambda-\zeta)^{-1}P_{t_0}+(\lambda-\zeta)^{-2}\mathfrak Q_1(t-t_0)\\
&+\Big(\frac12(\lambda-\zeta)^{-2}\mathfrak Q_2-(\lambda-\zeta)^{-1}\mathfrak S_2+ (\lambda-\zeta)^{-2}\tilde {\mathfrak R}(\zeta)\Big)(t-t_0)^2+o(t-t_0)^2.
\end{split}
\end{align}
The Laurent expansion \eqref{3.1ab} yields
\begin{equation}
\frac{-1}{2\pi\bfi}\int_{\gamma}\zeta(\lambda-\zeta)^{-2}R_{t_0}(\zeta)d\zeta=-S_{t_0}-\lambda S_{t_0}^2,
\end{equation}
and, consequently,
\begin{equation}\lb{integ1}
\frac{-1}{2\pi\bfi}\int_{\gamma}\zeta(\lambda-\zeta)^{-2}\tilde {\mathfrak R}(\zeta)d\zeta=-\mathfrak S_1-\lambda\mathfrak S_2.
\end{equation}
Multiplying both sides of \eqref{tau1} by $\frac{-1}{2\pi\bfi}\zeta$, integrating over $\gamma$, and using \eqref{integ1} we obtain  \eqref{A1}. 

Let us derive \eqref{A2} for $k=2$ (the case $k=3$ follows by conjugation). Using Taylor expansions \eqref{new1.43nn}, \eqref{pu1p} and denoting $\hatt{\mathfrak R}(\zeta):=(\tr R_{t_0}(\zeta))^*\dot Q_{t_0}J\tr S_{t_0}$  we arrive at
\begin{align}
&\tau_2(\zeta)=\big(\mathfrak S_0(t-t_0)+o(t-t_0)\big)(I-P_{t_0})\big((\lambda-\zeta)^{-1}\hatt{\mathfrak R}(\zeta)(t-t_0)+o(t-t_0)\big)P_{t_0}\\
&\hspace{1cm}\times \big(P_{t_0}-\frac{1}{2}\mathfrak S_2(t-t_0)^2+o(t-t_0)^2\big),
\end{align}
and, consequently, infer
\begin{align}\lb{tau2}
\tau_2(\zeta)&=(\lambda-\zeta)^{-1}\mathfrak S_0\hatt{\mathfrak R}(\zeta)(t-t_0)^2+o(t-t_0)^2.
\end{align}
Let us observe that
\begin{equation}
\frac{-1}{2\pi\bfi}\int_{\gamma}\zeta(\lambda-\zeta)^{-1}R_{t_0}(\zeta)d\zeta=-P_{t_0}+\lambda S_{t_0},\ S_{t_0} P_{t_0}=0,
\end{equation}
thus
\begin{equation}\lb{integ2}
\frac{-1}{2\pi\bfi}\int_{\gamma}\zeta(\lambda-\zeta)^{-1}\mathfrak S_0\hatt{\mathfrak R}(\zeta)d\zeta=\lambda\mathfrak S_2,
\end{equation}
where we used $S_{t_0} P_{t_0}=0$. Multiplying both sides of  \eqref{tau2} by $\frac{-1}{2\pi\bfi}\zeta$, integrating over $\gamma$ and using \eqref{integ2}we obtain
\begin{align}
\frac{-1}{2\pi \bfi}\int_{\gamma} \zeta \tau_2(\zeta)d\zeta=\lambda\mathfrak S_2(t-t_0)^2+o(t-t_0)^2.
\end{align}

Let us prove \eqref{A4}. Using the expansion of the resolvent \eqref{new1.43nn} to the first order, \eqref{pu1p} and its conjugate version, denoting $\mathfrak R(\zeta)=(\tr R_{t_0}(\zeta))^*\dot Q_{t_0}J \tr R_{t_0}(\zeta)$ we obtain
\begin{align}
&\tau_4(\zeta)=\big(\mathfrak S_0(t-t_0)+o(t-t_0)\big)((I-P_{t_0}) R_{t_0}(\zeta)(I-P_{t_0})+\mathfrak R(\zeta)(t-t_0)+o(t-t_0))\\
&\quad \times\big(\mathfrak S_0^*(t-t_0)+o(t-t_0)\big)=\mathfrak S_0(I-P_{t_0}) R_{t_0}(\zeta)(I-P_{t_0})\mathfrak S_0^*(t-t_0)^2+o(t-t_0)^2
\end{align}
Multiplying both sides of the above equation by $\frac{-1}{2\pi\bfi}\zeta$, integrating over $\gamma$ and noticing that the function $\zeta\mapsto (I-P_{t_0}) R_{t_0}(\zeta)(I-P_{t_0})$ is analytic with respect to $\zeta$ near $\lambda$ we obtain \eqref{A4}.\end{proof}

Finally, to prove the main assertion, \eqref{1.42n}, one combines \eqref{A1}--\eqref{A4} (add up all equations) together with \eqref{fourterms}, \eqref{hpt}. 
\end{proof}
We are now ready to present our next principal result, a version of what is sometimes called the Kato Selection Theorem, cf.\ \cite[Theorem II.5.11]{K80}, that allows one to compute the second order Taylor expansion of the appropriately chosen eigenvalue curves bifurcating from a multiple eigenvalue.
\begin{theorem}\lb{EVExp}
Assuming Hypotheses \ref{hyp1.3bis} and \ref{hyp3.1o}, we let $W_t=P_{t_0}U_t^{-1}\cA_tU_tP_{t_0}$ and suppose that $\dot W_{t_0}$, considered as an operator in $\ran(P_{t_0})$, has $m'\leq m$  distinct eigenvalues denoted by $\mu_1, ... , \mu_{m'}$. We  denote the corresponding to $\mu_i$ eigenprojection of  $\dot W_{t_0}$ in $\ran(P_{t_0})$ by $P^{(i)}_{t_0}$, $1\leq  i\leq m'$. Furthermore, for each $i\leq m'$ we suppose that $P^{(i)}_{t_0}\ddot W_{t_0} P^{(i)}_{t_0}$ considered as an operator in $\ran(P^{(i)}_{t_0})$ has $m'_i\leq m$  distinct eigenvalues denoted by $\nu_{i,1},..., \nu_{i, m'_i} $.  Then for $t$ near $t_0$ there is a labeling $\lambda_{i, k}(t)$ of the eigenvalue curves of the operator $\cA_tP_t$ or $W_t$ such that
\begin{align}\lb{razlozhenie}
\lambda_{i, k}(t)=\lambda+\mu_i(t-t_0)+\frac 12\nu_{i, k}(t-t_0)^2+o(t-t_0)^2,\  k=1,.., m'_i, i=1,..., m'. 
\end{align} 
Moreover, there is an orthonormal set of eigenvectors $u_{i,k}$ of the operators $P^{(i)}_{t_0}\ddot W_{t_0} P^{(i)}_{t_0}$ (and therefore of the operators  $\dot{W}_{t_0}$, $W_{t_0}$ and $\cA_{t_0}$), 
\[\{u_{i,k}\in\ran(P^{(i)}_{t_0})):  k=1,.., m'_i, \, i=1,..., m'\}\subset \ran(P_{t_0}),\] such that 
\begin{align}
\begin{split}\lb{razl}
&\lambda_{i, k}(t)=\lambda
+\omega(\dot Q_{t_0}\tr u_{i,k}, \tr u_{i,k})(t-t_0)\\
&\quad+\frac 12\Big(\omega(\ddot Q_{t_0}\tr u_{i,k}, \tr u_{i,k})-2\langle \tr (\tr S_{t_0})^*\dot Q_{t_0}J\tr P_{t_0}u_{i,k}, \dot Q_{t_0}J\tr P_{t_0}u_{i,k}\rangle_{\bndra}\Big)(t-t_0)^2\\
&\quad\quad+o(t-t_0)^2,\   k=1,.., m'_i,\, i=1,..., m', 
\end{split}
\end{align} 
where $\tr\in\cB(\cH_+,\mathfrak{H}\times\mathfrak{H})$ is the trace operator, the symplectic form $\omega$ is defined in \eqref{5.3},  $S_{t_0}$ is the reduced resolvent of the operator $\cA_{t_0}$ at $\lambda$.
\end{theorem}
\begin{proof}
 Taylor's expansion \eqref{1.42n} yields
\begin{align}
\frac{1}{(t-t_0)}\Big(\frac{1}{(t-t_0)}\big(W_t-W_{t_0}\big)-\dot W_{t_0}\Big)\underset{t\rightarrow t_0}{=}\frac12\ddot W_{t_0}+o(1).
\end{align}
Multiplying this by $P^{(i)}_{t_0}$ from both sides and using the identities 
\begin{equation}
P^{(i)}_{t_0}W_{t_0}P^{(i)}_{t_0}=\lambda I_{\ran P_{t_0}}, \quad P^{(i)}_{t_0}\dot W_{t_0}P^{(i)}_{t_0}=\mu_i I_{\ran P^{(i)}_{t_0}},
\end{equation}
we obtain 
\begin{align}
\frac{1}{(t-t_0)}\Big(\frac{1}{(t-t_0)}\big(P^{(i)}_{t_0}W_tP^{(i)}_{t_0}-\lambda\big)-\mu_i \Big)\underset{t\rightarrow t_0}{=}\frac12P^{(i)}_{t_0}\ddot W_{t_0}P^{(i)}_{t_0}+o(1).\lb{oplhs}
\end{align}
By continuity of the spectrum, see, e.g., \cite[Theorem VIII.24]{RS78}, every point in the set
\[\spec\left(\frac 12 P^{(i)}_{t_0}\ddot W_{t_0}P^{(i)}_{t_0}\right)=\left\{\frac 12\nu_{i,1},..., \frac 12\nu_{i,m'_i}\right\}\] 
can be approximated by an eigenvalue of the operator in the left-hand side of \eqref{oplhs} for $t$ near $t_0$. Such an eigenvalue is of the form
\begin{equation}
\frac{1}{(t-t_0)}\Big(\frac{1}{(t-t_0)}\big(\lambda_{i,k}(t)-\lambda\big)-\mu_i \Big), \text{ for some choice of $\lambda_{i,k}(t)\in\spec(\cA_tP_t|_{\ran(P_t)})=\spec(W_t)$.}
\end{equation}
That is, we have
\begin{align}
\frac{1}{(t-t_0)}\Big(\frac{1}{(t-t_0)}\big(\lambda_{i,k}(t)-\lambda\big)-\mu_i \Big)\underset{t\rightarrow t_0}{=}\frac12 \nu_{i,k}+o(1),
\end{align}
hence, \eqref{razlozhenie} holds as asserted. 

To prove \eqref{razl}, let $u_{i,k}\in \ran(P^{(i)}_{t_0})$ be the normalized eigenvectors of $\ddot W_{t_0}$ such that $\ddot W_{t_0} u_{i,k}=\nu_{i,k}u_{i,k}$. Since  $u_{i,k}\in \ran(P^{(i)}_{t_0})$ we have $\dot W_{t_0}u_{i,k}=\mu_iu_{i,k}$. Using these identities and \eqref{T1}, \eqref{T2} we arrive at

\begin{align}
\begin{split}\lb{munu}
&\mu_i=\langle \dot W_{t_0}u_{i,k},u_{i,k}\rangle_{\cH}=\omega(\dot Q_{t_0}\tr u_{i,k}, \tr u_{i,k}),\\
&\nu_{i,k}=\langle \ddot W_{t_0}u_{i,k},u_{i,k}\rangle_{\cH} =\omega(\ddot Q_{t_0}\tr u_{i,k}, \tr u_{i,k})-2 \langle\tr (\tr S_{t_0})^* \dot Q_{t_0}J\tr u_{i,k}, \dot Q_{t_0}J\tr P_{t_0} u_{i,k}\rangle_{\cH}.
\end{split}
\end{align}
Finally, combining \eqref{munu} and \eqref{razlozhenie} we obtain \eqref{razl}
\end{proof}

In applications, e.g. such as quantum graphs, spectral theory of differential operators, one often describes a Lagrangian plane  $\cG\in \Lambda(\bndra)$ by means of two bounded operators  $X,Y\in\cB(\mathfrak H)$, satisfying $XY^*=YX^*$, $0\not\in \spec(XX^*+YY^*)$, via  
\[\cG=\{(u,v)\in\bndra: Xu+Yv=0\}=\ker (Z),\, \text{ where $Z=[ X, Y]\in\cB(\bndra,\mathfrak{H})$}.\] 
It is therefore natural to derive the second order Taylor expansion in this setting, which is done in the following theorem. 
\begin{theorem}\lb{prop2.9} Under hypotheses of Theorem \ref{EVExp}, suppose that the Lagrangian planes $\ran(Q_t)=\tr(\dom(\cA_t))$ are given by $\ran(Q_t)=\ker Z_t$ where  $Z_t:=[X_t, Y_t]$ and $X_t, Y_t\in\cB(\mathfrak H)$, $X_tY_t^*=Y_tX_t^*$, $0\not\in\spec(X_tX^*_t+Y_tY^*_t)$, satisfy 
	\begin{align}
	&Z_t\underset{t\rightarrow t_0}{=}Z_{t_0}+\dot Z_{t_0}(t-t_0)+\frac12\ddot Z_{t_0}(t-t_0)^2+o(t-t_0)^2\text{\ in\ }\cB(\bndra).
	\end{align}
Then the eigenvalue curves $\lambda_{i, k}(t)$, defined as in Theorem \ref{EVExp}, bifurcating from the isolated eigenvalue $\lambda\in\spec (\cA_{t_0})$ admit the second order Taylor expansion
\begin{align}
\begin{split}\lb{XYsecodr}
\lambda_{i,k}(t)\underset{t\rightarrow t_0}{=}&\lambda+ \langle Z_{t_0} J\dot Z^*_{t_0}\phi_{i,k}, \phi_{i,k} \rangle_{\mathfrak{H}}(t-t_0)\\
&\quad+\frac{1}{2} \left(\langle Z_{t_0} J\ddot Z^*_{t_0}\phi_{i,k}, \phi_{i,k} \rangle_{\mathfrak{H}}-2\langle \tr (\tr S)^*\dot Z_{t_0}^*\phi_{i,k}, \dot Z_{t_0}^* \phi_{i,k} \rangle_{\bndra}\right)(t-t_0)^2\\
&\quad\qquad+o(t-t_0)^2,
\end{split}
\end{align}  
where $S_{t_0}$ is the reduced resolvent of $\cA_{t_0}$ at $\lambda$, and we introduce notation 
\[\phi_{i,k}= (Z_{t_0}Z_{t_0}^*)^{-1}Z_{t_0}J\tr u_{i,k}\]
for  $u_{i,k}$ given in Theorem \ref{EVExp}. 
\end{theorem}
\begin{proof}
For brevity, we denote $u=u_{i,k}$, $\phi:=\phi_{i,k}$ and suppress dependence of $P_{t_0}$, $Z_{t_0}$, $Q_{t_0}$ and their derivatives on $t_0$, e.g., $\ddot Z_{t_0}$ will be denoted by $\ddot Z$. Let us record some relevant identities. First, the orthogonal projection $Q$ onto the kernel of $Z$, i.e. onto the Lagrangian plane $\tr(\dom(\cA))$, can be written as
\begin{equation}\lb{projZ}
Q=-JZ^*(ZZ^*)^{-1}ZJ,
\end{equation}
thus we have $\tr u=Q\tr u$, and, by $\phi= (ZZ^*)^{-1}ZJ\tr u$, one also has
\begin{align}
\begin{split}
&\tr u=-JZ^*\phi. \lb{tujzs}
\end{split}
\end{align}
In addition, the identity $XY^*=Y X^*$ yields $ZJZ^*=0,$ therefore, by differentiating we arrive at
\begin{align}
&\dot Z J Z^*=- Z J\dot Z^*, \lb{dz}\\
&\ddot Z J Z^*=-2\dot Z J\dot Z^*- Z J \ddot Z^*.\lb{ddz}
\end{align}
Furthermore, since $\ran(Q)=\ker(Z)$ we have $ZQ=0$, hence,
\begin{align}
&\dot ZQ+Z\dot Q=0\lb{dzq},\\
&Q\dot Z^*+\dot QZ^*=0\lb{qdz}.
\end{align}
Also, since $QJ+JQ=J$ we get
\begin{equation}\lb{dqj}
\dot QJ=-J\dot Q,\ \ddot QJ=-J\ddot Q. 
\end{equation}
Let us rewrite the projection $Q$ onto the Lagrangian plane as $Q=\alpha\beta\gamma$ with
\begin{align}
&\alpha:=-JZ^*\in\cB(\mathfrak H, \bndra),\ \beta:= (ZZ^*)^{-1}\in\cB(\mathfrak H),\ \gamma:=ZJ\in\cB(\bndra, \mathfrak{H}).
\end{align}
We proceed by introducing three more auxiliary identities \eqref{dqjtu}, \eqref{dqtu}, \eqref{ident3} below. One has
\begin{equation}\lb{dqjtu}
\dot QJ\tr u=-Q\dot Z^* \phi.
\end{equation}
To prove this, we note that $\gamma J\tr u=0$, therefore
\begin{align}
\dot QJ\tr u&=\alpha \beta \dot \gamma J\tr u=-\alpha \beta\dot Z\tr u\underset{\eqref{tujzs}}{=}\alpha \beta\dot ZJZ^*\phi\underset{\eqref{dz}}{=}-\alpha \beta ZJ\dot Z^*\phi=-Q\dot Z^*\phi.
\end{align}
We also have
\begin{equation}\lb{dqtu}
\dot Q \tr u=(1-Q)\tr \dot P u.
\end{equation}
One arrives at this identity by differentiating both sides of the equation $QTP=TP$ and noting that $Pu=u$. 
Next, we record the identity 
\begin{equation}\lb{ident3}
\dot Q \tr u=(1-Q)\tr (\tr S)^*\dot Q J\tr u,
\end{equation}
which follows from \eqref{dqtu}, explicit formula for $\dot P$ in \eqref{pone}, and the fact that the second term in the RHS of \eqref{pone} vanishes since $S_{t_0}u=0$. 

We are now ready to prove \eqref{XYsecodr}. To that end we employ \eqref{razl} together with the auxiliary identities \eqref{tujzs}-\eqref{ident3}. It suffices to show that 
\begin{align}
&\omega(\dot Q\tr u, \tr u)=\langle ZJ\dot Z^*\phi,\phi\rangle_{\mathfrak{H}},\lb{omdq}\\
&\omega(\ddot Q\tr u, \tr u)-2\langle \tr (\tr S)^*\dot QJ\tr u, \dot QJ\tr u\rangle_{\bndra}=\\
&\hspace{2cm}=\langle ZJ\ddot Z^*\phi,\phi\rangle_{\mathfrak{H}}-2\langle \tr (\tr S)^*\dot Z^*\phi, \dot Z^*\phi\rangle_{\bndra}\lb{omddq}.
\end{align}

Let us first derive \eqref{omdq}. One has
\begin{align}
\omega(\dot Q\tr u, \tr u)&=\langle J\dot Q \tr u, \tr u\rangle_{\bndra}\underset{\eqref{dqj}}{=}-\langle \dot Q J\tr u, \tr u\rangle_{\bndra}\underset{\eqref{dqjtu}}{=}\langle Q\dot Z^* \phi, \tr u\rangle_{\bndra}\\&=\langle \dot Z^* \phi, Q\tr u\rangle_{\bndra}
=\langle \dot Z^* \phi, \tr u\rangle_{\bndra}\underset{\eqref{tujzs}}{=}\langle \dot Z^* \phi, -JZ^*\phi\rangle_{\bndra}=\langle Z J\dot Z^* \phi,\phi\rangle_{\mathfrak{H}}.
\end{align}
The derivation of \eqref{omddq} is more subtle.  Our immediate objective is to derive the following,
\begin{align}
&\omega(\ddot Q\tr u, \tr u)=\langle ZJ\ddot Z^*\phi, \phi\rangle_{\mathfrak H}-2\omega(Q\dot Z^*\phi, \dot Z^*\phi),\lb{ident1}\\
&\omega(Q\dot Z^*\phi, \dot Z^*\phi)=\langle \tr (\tr S)^*\dot Z^*\phi, (1-Q)\dot Z^*\phi\rangle_{\bndra}, \lb{ident2}\\
&\langle \tr (\tr S)^*\dot QJ\tr  u, \dot QJ\tr  u\rangle_{\bndra}=\langle \tr  (\tr  S)^*\dot Z^*\phi, Q\dot Z^*\phi\rangle_{\bndra},\lb{ident3p}.
\end{align}
since \eqref{omddq} readily follows from \eqref{ident1}--\eqref{ident3p}.
To to begin the proof of \eqref{ident1} we differentiate $Q=\alpha\beta\gamma$ twice. Since $\gamma J\tr  u=0$ and infer
\begin{equation}
\ddot Q J \tr  u=\alpha\beta\ddot \gamma J\tr  u+2(\dot \alpha\beta+\alpha\dot\beta)\dot \gamma J\tr u.
\end{equation}
Plugging this into the left-hand side of \eqref{ident1} we arrive at
\begin{align}\lb{vsp1}
&\omega(\ddot Q\tr u, \tr u)=-\langle \ddot QJ\tr u,\tr u\rangle_{\bndra}\\
&\quad=-\langle \alpha\beta\ddot \gamma  J\tr u,\tr u\rangle_{\bndra}-2\langle(\dot \alpha\beta+\alpha\dot\beta)\dot \gamma J\tr u, \tr u\rangle_{\bndra}.
\end{align}
Let us re-write both terms in the right-hand side of \eqref{vsp1}. The first term is
\begin{align}
\begin{split}\lb{abgddot}
&-\langle \alpha\beta\ddot \gamma  J\tr u,\tr u\rangle_{\bndra}\underset{\eqref{tujzs}}{=}-\langle \alpha\beta \ddot Z J Z^*\phi ,\tr u\rangle_{\bndra}\\
&\underset{\eqref{ddz}}{=}2\langle \alpha\beta \dot Z J \dot Z^*\phi ,\tr u\rangle_{\bndra}+\langle \alpha\beta  Z J \ddot Z^*\phi ,\tr u\rangle_{\bndra}\\
&=2\langle (ZZ^*)^{-1} \dot Z J \dot Z^*\phi , ZZ^*\phi\rangle_{\mathfrak{H}}+\langle \alpha\beta \gamma \ddot Z^*\phi ,\tr u\rangle_{\bndra}\\
&=2\langle   J \dot Z^*\phi , \dot Z^*\phi\rangle_{\bndra}+\langle Q  \ddot Z^*\phi ,Tu\rangle_{\bndra}\underset{\eqref{tujzs}}{=}\langle    ZJ\ddot Z^*\phi ,\phi\rangle_{\mathfrak H},
\end{split}
\end{align}
where we used $\langle   J \dot Z^*\phi , \dot Z \phi\rangle_{\bndra}=0$. In order to re-write the second term in the right-hand side of \eqref{vsp1} we first notice that $\beta=(ZZ^*)^{-1}$ yields
\[\dot \beta =-\beta(\dot Z  Z^*+Z\dot Z^*)\beta,\]
thus 
\begin{align}
\begin{split}\lb{abgdots}
&\langle(\dot \alpha\beta+\alpha\dot\beta)\dot \gamma J\tr u, \tr u\rangle_{\bndra}\\
&\quad=\langle\dot \alpha\beta\dot \gamma J\tr u, \tr u\rangle_{\bndra}-\langle \alpha\beta\dot Z Z^*\beta \dot \gamma J\tr u, \tr u\rangle_{\bndra}-\langle \alpha\beta Z \dot Z^*\beta\dot \gamma J\tr u, \tr u\rangle_{\bndra}\\
&\quad=\langle\dot \alpha\beta\dot \gamma J\tr u, \tr u\rangle_{\bndra}-\langle \alpha\beta\dot Z Z^*\beta \dot \gamma J\tr u, \tr u\rangle_{\bndra}+\langle \alpha\beta ZJ J \dot Z^*\beta\dot \gamma J\tr u, \tr u\rangle_{\bndra}\\
&\quad=\langle\dot \alpha\beta\dot \gamma JTu, Tu\rangle_{\bndra}-\langle \alpha\beta\dot Z Z^*\beta \dot \gamma J\tr u, \tr u\rangle_{\bndra}-\langle \alpha\beta\gamma \dot\alpha\beta\dot \gamma J\tr u, \tr u\rangle_{\bndra}\\
&\quad=\langle\dot \alpha\beta\dot \gamma J\tr u, \tr u\rangle_{\bndra}-\langle \alpha\beta\dot Z Z^*\beta \dot \gamma J\tr u, \tr u\rangle_{\bndra}-\langle \dot\alpha\beta\dot \gamma J\tr u, Q\tr u\rangle_{\bndra}\\
&=\quad-\langle \alpha\beta\dot Z Z^*\beta \dot \gamma J\tr u, \tr u\rangle_{\bndra},
\end{split}
\end{align}
where we used $Q=\alpha\beta\gamma$, $Q\tr u=\tr u$. Now, let us further simplify the last expression. One has
\begin{align}
\begin{split}\lb{abzdots}
-\langle \alpha\beta&\dot Z Z^*\beta \dot \gamma J\tr u, \tr u\rangle_{\bndra}=-\langle \dot Z Z^*(Z Z^*)^{-1} \dot Z J J\tr u, (ZZ^*)^{-1}ZJ\tr u\rangle_{\mathfrak H}\\
&\quad\underset{\eqref{tujzs}}{=}-\langle \dot Z Z^*(Z Z^*)^{-1} \dot Z J Z^*\phi, \phi \rangle_{\mathfrak H}\underset{\eqref{dz}}{=}\langle \dot Z Z^*(Z Z^*)^{-1}  Z J \dot Z^*\phi, \phi \rangle_{\mathfrak H}\\
&\quad=\langle J(-JZ)^*(Z Z^*)^{-1}  Z J \dot Z^*\phi, \dot Z^* \phi \rangle_{\bndra}=\omega(Q\dot Z^*\phi,\dot Z^*\phi).
\end{split}
\end{align}
Combining \eqref{vsp1}, \eqref{abgdots} and \eqref{abzdots} one infers \eqref{ident1} as asserted. 

Next, we prove \eqref{ident2}. One has
\begin{align}
&\omega(Q\dot Z^*\phi,\dot Z^*\phi)=\langle JQ\dot Z^*\phi, \dot Z^*\phi \rangle_{\bndra}\underset{\eqref{qdz}}{=}-\langle J\dot Q Z^*\phi, \dot Z^*\phi \rangle_{\bndra}\underset{\eqref{dqj}}{=}\langle \dot QJ Z^*\phi, \dot Z^*\phi \rangle_{\bndra}\\
&\quad=-\langle \dot Q\tr u, \dot Z^*\phi \rangle_{\bndra}\underset{\eqref{ident3}}{=}-\langle (1-Q)\tr (\tr S)^*\dot Q J\tr u, \dot Z^*\phi \rangle_{\bndra}\\
&\quad\underset{\eqref{dqjtu}}{=}\langle (1-Q)\tr (\tr S)^*Q\dot Z^*\phi, \dot Z^*\phi \rangle_{\bndra}\\
&\quad=\langle \tr (Q\tr S)^*\dot Z^*\phi,(1-Q) \dot Z^*\phi \rangle_{\bndra}=\langle \tr (\tr S)^*\dot Z^*\phi,(1-Q) \dot Z^*\phi \rangle_{\bndra},
\end{align}
where we used $Q\tr S=\tr S$ which holds since $\ran(S)\subset \ran(Q)$. 

To complete the proof of the theorem we derive \eqref{ident3p}. To that end, we have 
\begin{align}
&\langle \tr (\tr S)^*\dot QJ\tr u, \dot QJ\tr u\rangle_{\bndra}\underset{\eqref{dqjtu}}{=}\langle \tr (\tr S)^*Q\dot Z^*\phi, Q\dot Z^*\phi\rangle_{\bndra}\\
&\quad= \langle \tr (Q\tr S)^*\dot Z^*\phi, Q\dot Z^*\phi\rangle_{\bndra}=\langle \tr (\tr S)^*\dot Z^*\phi, Q\dot Z^*\phi\rangle_{\bndra},
\end{align}
where, as before, we used $Q\tr S=\tr S$. 
\end{proof}

\begin{corollary}\lb{cor211} Under the hypotheses and general setting of Theorem \ref{prop2.9} we assume that $X_t=\Theta_t$, $Y_t=-I$, that is, we assume that $\cA_t$ is a Robin-type self-adjoint extension of $A$ with $\dom(\cA_t)=\{u\in\cH_+: \Gamma_1u=\Theta_t\Gamma_0u \}$, and that $\Theta_t=\Theta_t^*$. Then one has
	\begin{align}
	\begin{split}\lb{XYsecodrnew}
	\lambda_{i,k}(t)\underset{t\rightarrow t_0}{=}&\lambda+\langle\dot \Theta\Gamma_0u_0, \Gamma_0 u_0 \rangle_{{\mathfrak{H}}}(t-t_0)\\
	&\quad+\frac{1}{2} \left(\langle\big(\ddot \Theta_{t_0}-2\dot \Theta_{t_0} \Gamma_0 (\Gamma_0S_{t_0})^*\dot \Theta_{t_0}\big)\Gamma_0 u_{i,k},\Gamma_0 u_{i,k} \rangle_{\mathfrak{H}}\right)(t-t_0)^2\\
	&\quad+o(t-t_0)^2.
	\end{split}
	\end{align}  
\end{corollary}
\begin{proof}Formula \eqref{XYsecodrnew} follows directly from \eqref{XYsecodr} with $Z_t=[\Theta_t, -I]$.
\end{proof}
\begin{example}[Matrix Robin Laplacian] Let $V\in L^{\infty}([a,b], \bbC^{n\times n})$ be a matrix valued bounded function with $V(x)=(V(x))^*$, $x\in[a,b]$, and let $A=-\frac{d^2}{dx^2}+V$ be the minimal operator with $\dom(A)=H^2_0([a,b]; \bbC^n)$. Given $u\in H^2([a,b], \bbC^n)$ we denote $\gaD u=[u(a), u(b)]^{\top}$, $\gaN u=[u(a), -u'(b)]^{\top}$ and view $\gaD, \gaN$ as bounded linear operators in $\cB(H^2([a,b],\bbC^{2n} );\bbC^{2n})$.  Let $\Theta\in C^2([0,1], \bbC^{2n\times 2n})$ be a matrix-valued function with $\Theta_t=\Theta_t^*$, $t\in[0,1]$, and define a one-parameter family  of Robin extensions of $A$ by the formula
	\begin{equation}
	\dom(\cA_t)=\{u\in H^2([a,b]; \bbC^n): \gaN u=\Theta_t \gaD u\}.
	\end{equation}
	Suppose that $\lambda\in\bbR$ is a simple eigenvalue of $\cA_{t_0}$ for some $t_0$ and suppose that $u_0\in\ker(\cA_{t_0}-\lambda)$ has unit norm. Then there is an eigenvalue curve bifurcating from $\lambda$ and obeying the following asymptotic expansion,
	\begin{align}\begin{split}\lb{1dRobin}
	\lambda(t)=
	\lambda&+\left\langle\dot \Theta_{t_0} \gaD u, \gaD u \right\rangle_{\C^{2n}}(t-t_0)\\
	&+\frac12\left\langle\big(\ddot \Theta_{t_0}-2\dot \Theta_{t_0} \gaD(\gaD S)^*\dot \Theta_{t_0}\big)  \gaD u,\gaD u\right\rangle_{\bbC^{2n}}(t-t_0)^2+o(t-t_0)^2, t\rightarrow t_0.
	\end{split}
	\end{align}
\end{example}

\begin{remark}\lb{rem2.14}Our main results, Theorem \ref{theorem2.2},\ref{EVExp} and Theorem \ref{prop2.9} can be extended to a  more general setting in which the self-adjoint extensions are subject to further additive perturbations  $\cA_t+V_t$. To elaborate on this case, we augment Hypothesis \ref{hyp1.3bis} by fixing a one-parameter family of bounded self-adjoint operators  $t\mapsto V_t\in\cB(\cH)$ with
	\begin{equation}
	V_t\underset{t\rightarrow t_0}{=}V_{t_0}+\dot V_{t_0}(t-t_0)+\frac12\ddot V_{t_0}(t-t_0)^2+o(t-t_0)^2\text{\ in\ }\cB(\cH).
	\end{equation}
 Let us introduce $H_t:=\cA_t+V_t$, denote its resolvent by $R_t(\zeta):=(H_t-\zeta)^{-1}\in\cB(\cH)$ for $\zeta\not\in \Sp(H_t)$, and its reduced resolvent at the eigenvalue $\lambda$ by $S_{t_0}$. Keeping general setting of Theorem \ref{EVExp} with $\cA_t$ replaced by $H_t$ we record the first and second order terms in the Taylor expansion of the eigenvalue curves. The first order term is given by
\begin{align}
\dot \lambda_{i,k}=\langle\dot V_{t_0}u_{i,k}, u_{i,k}\rangle_{\cH}+\omega(\dot Q_{t_0}\tr u_{i,k}, \tr u_{i,k}).
\end{align} 
The second order term is given by
\begin{align}
\ddot \lambda_{i,k}&=\frac 12\Big(\langle\ddot V_{t_0}u_{i,k}, u_{i,k}\rangle_{\cH}+\omega(\ddot Q_{t_0}\tr u_{i,k}, \tr u_{i,k})\Big)-\langle \tr (\tr S_{t_0})^*\dot Q_{t_0}J\tr P_{t_0}u_{i,k}, \dot Q_{t_0}J\tr P_{t_0}u_{i,k}\rangle_{\bndra}\\
&+2\langle \tr S \dot V_{t_0}Pu_{i,k}, \dot Q_{t_0}J\tr P_{t_0}u_{i,k}\rangle_{\bndra}-\langle \dot V_{t_0}S_{t_0}\dot V_{t_0}u_{i,k}, u_{i,k}\rangle_{\cH}. 
\end{align}
The derivation relies on Krein's formula from \cite{LSHadamard},
\begin{align}
R_{t}(\zeta)- R_{t_0}(\zeta)=R_t({\zeta})(V_{t_0}-V_t)R_{t_0}(\zeta)+(TR_t(\zeta))^*(Q_t-Q_{t_0})JQ_{t_0}TR_{t_0}(\zeta), 
\end{align}
and resolvent expansions arguments presented above. 
\end{remark}

\bibliography{mybib}
\bibliographystyle{siam}

\end{document}